\documentclass[a4paper,12pt, reqno]{amsart}
\usepackage[]{hyperref}
\usepackage{a4wide}
\usepackage{amsthm}
\usepackage{amssymb}
\usepackage{palatino}
\usepackage{enumitem}
\usepackage{tcolorbox}
\usepackage{amssymb}
\newtheorem{thm}{Theorem}

\newtheorem{corollary}{Corollary}[section]
\newtheorem{lemma}{Lemma}

\makeatletter
\newcommand{\Mod}[1]{\ (\mathrm{mod}\ #1)}
\newcommand*{\rom}[1]{\expandafter\@slowromancap\romannumeral #1@}

\def\blfootnote{\gdef\@thefnmark{}\@footnotetext}
\makeatother
\def\house#1{\setbox1=\hbox{$\,#1\,$}
\dimen1=\ht1 \advance\dimen1 by 2pt \dimen2=\dp1 \advance\dimen2 by 2pt
\setbox1=\hbox{\vrule height\dimen1 depth\dimen2\box1\vrule}%
\setbox1=\vbox{\hrule\box1}%
\advance\dimen1 by .4pt \ht1=\dimen1
\advance\dimen2 by .4pt \dp1=\dimen2 \box1\relax}

\begin{document}
\title{Linear Congruences and a Conjecture of Bibak}
\author[Babu]{C. G. Karthick Babu} \email{cgkarthick24@gmail.com}
\author[Bera]{Ranjan Bera}
\email{ranjan.math.rb@gmail.com}
\author[Sury]{B. Sury}
\email{surybang@gmail.com}
\address{Statistics and Mathematics Unit, Indian Statistical Institute, R.V. College Post, Bangalore-560059, India.}

\date{}

\maketitle

\begin{abstract}
We address three questions posed by Bibak \cite{KB20}, and
generalize some results of Bibak, Lehmer and K G Ramanathan on
solutions of linear congruences $\sum_{i=1}^k a_i x_i \equiv b \Mod{n}$. In particular, we obtain explicit expressions for the number
of solutions where $x_i$'s are squares modulo $n$. In addition, we
obtain expressions for the number of solutions with order
restrictions $x_1 \geq \cdots \geq x_k$ or, with strict order
restrictions $x_1> \cdots > x_k$ in some special cases. In these results, the expressions for the number of solutions involve Ramanujan sums and are obtained using  their properties.
\end{abstract}

\section{Introduction and Statements of Main Results}

\noindent Let $a_{1}, \dots, a_{k}, b$ be integers and $n$ be a
positive integer. A linear congruence in $n$ unknowns $x_{1}, \dots,
x_{k}$ is a congruence of the form
\begin{equation}\label{linear cong}
a_{1}x_{1}+\dots+a_{k}x_{k} \equiv b \Mod{n}.
\end{equation}
More than a century ago, D. N. Lehmer \cite{DNL13}, proved that a
linear congruence represented by \eqref{linear cong} has a
solution $\langle x_{1}, \dots, x_{k}\rangle \in \mathbb{Z}_{n}^{k}$
if and only if $\ell \mid b$, where $\ell=(a_{1}, \dots, a_{k}, n)$. Here and henceforth the notation $(u_{1}, \dots, u_{k})$  denotes the GCD of integers $u_{1}, \dots, u_{k}$. Further, if this condition is satisfied, then there are $\ell
n^{k-1}$ solutions. Over the years, solutions of the linear
congruence \eqref{linear cong}, which are subject to different types
of restrictions such as GCD restrictions $(x_{i}, n)=t_{i}$ ($1 \leq
i \leq k$), for prescribed divisors $t_1, \cdots, t_k$ of $n$, have
been extensively investigated. \vspace{2mm}

\noindent In a different direction, order-restricted solutions $x_1
\geq \dots \geq x_k$ of the linear congruence represented by \eqref{linear cong}, seem to have been studied for cryptographic
applications. In 1962, Riordan \cite{JR62} derived an explicit
formula for order-restricted solutions when $a_i=1$ for all $i$,
$b=0$ and $n=k$. More recently, Bibak \cite{KB20} extended Riordan's
result; specifically to the case where $a_i=1$ for all $i$ as
before, but allowing arbitrary integers $k$ and $b$. The vast
literature (see, \cite{BBVRL17, AB26, CE55, NV54, HR25}) on these
topics, bears witness to their applicability in diverse fields such
as cryptography, coding theory, combinatorics, and computer science.
For a comprehensive overview of these applications and further
insights, one can refer to \cite{BKV16, BKV19, BBVRL17, BKVT18,
JW72}. \vskip 2mm

\noindent In this paper, we answer some questions posed by Bibak,
and generalize some results of Bibak, Lehmer and K G Ramanathan (see
Theorems \ref{square thm}, \ref{strict ord thm}, \ref{perm thm} and
Corollary \ref{square corollary}). The expressions naturally involve
Ramanujan sums. \vskip 2mm

\noindent In the last section of \cite{KB20}, Bibak posed the
following three interesting problems. The first one asks for
solutions $x_i$ that are squares modulo $n$ for \eqref{linear cong}.
We solve this problem (Theorems \ref{square thm}) - subtleties about square
solutions are explained in section \ref{sqar solns-subt}. A second problem posed by
Bibak is to study the general case of strictly ordered solutions
$x_{1} > \dots
> x_{k}$ of \eqref{linear cong}, for which a
special case is addressed by us in Theorem \ref{strict ord thm}
stated below. Another problem asks for an explicit formula for the
number of order-restricted solutions $x_1 \geq \cdots \geq x_k$.
Theorem \ref{perm thm} below can be considered a partial answer
towards that, and is a mild generalization of Bibak's result. \vskip
3mm

\begin{thm}\label{square thm}
Let $S_{n}(b; a_{1}, \dots, a_{k})$ denote the number of square
solutions of \eqref{linear cong}. Assume $n$ is an odd positive
integer, having a prime factorization $n=p_{1}^{\ell_{1}}\cdots
p_{r}^{\ell_{r}}$. Then, we have
\begin{align*}
S_{n}(b; a_{1}, \dots, a_{k})=&\frac{1}{n} \prod_{q=1}^{r}
\bigg(\sum_{m=1}^{p^{\ell_{q}}}
e\bigg(\frac{-bm}{p^{\ell_{q}}}\bigg)+\sum_{\substack{K \subset \{1,
\dots, k\} \\ K \neq \phi}}\frac{1}{2^{|K|}} S_{K}\bigg),
\end{align*}
where
\begin{equation}\label{SK eq}
S_{K}=\sum_{m=1}^{p^{\ell_{q}}}e\bigg(\frac{-bm}{p^{\ell_{q}}}\bigg)\prod_{i
\in K} \bigg(\sum_{\substack{j_{i}=0 \\ j_{i} \equiv 0 \Mod
2}}^{\ell_{q}-1} \bigg(
C_{p^{\ell_{q}-j_{i}}}(a_{i}m)+\varepsilon_{p}
\bigg(\frac{a_{i}m/p^{\ell_{q}-j_{i}-1}}{p}\bigg)
\varrho_{j_{i}}(a_{i}m) p^{\ell_{q}-j_{i}-\frac{1}{2}}\bigg)\bigg),
\end{equation}
\begin{equation}\label{coprime char fn}
    \varrho_{j_{i}}(x)= \begin{cases}
    1 ,& \text{if } (x, p^{\ell_{q}-j_{i}})=p^{\ell_{q}-j_{i}-1} ,\\
     0, & \text{otherwise}.
    \end{cases}
\end{equation}
\end{thm}

\noindent Here $C_{n}(b)$ denotes the Ramanujan sum as before, and
the other notations appearing are  $\bigg(\frac{\cdot}{p}\bigg)$,
the Legendre symbol modulo $p$, $e(x) = e^{2i \pi x}$ and
$\varepsilon_n$ is the sign of the Gauss sum given by
\begin{equation}\label{epsilon-p defn}
\varepsilon_{n}=
\begin{cases}
1 \quad  \text{if} \ n \equiv 1 \Mod{4},\\
i \quad   \text{if} \ n \equiv 3 \Mod{4}.
\end{cases}
\end{equation}

\noindent In the special case when $n = p^{\ell}$, $(a_{i},
n)=1=(b,n)$, for all $i$, we have the following simplified expression for $S_K$'s in $S_{n}(b; a_{1}, \dots, a_{k})$:
\vskip 2mm

\begin{corollary}\label{square corollary}

\begin{align*}
S_{K}
= \bigg(\sum_{\substack{j=0 \\ j \equiv 0 \Mod 2}}^{\ell-1}  \phi({p^{\ell-j}})\bigg)^{|K|}+\sum_{m=1}^{p-1}e\bigg(\frac{-bm}{p}\bigg)\bigg(-p^{\ell-1}+\sum_{\substack{j=2 \\ j \equiv 0 \Mod 2}}^{\ell-1}\phi(p^{\ell-j})+\varepsilon_{p}\bigg(\frac{m}{p}\bigg)p^{\ell-\frac{1}{2}}\bigg)^{|K|}.
\end{align*}
\end{corollary}
\vskip 2mm

\noindent   For the number $N_n(k,a,b)$ of solutions $x_1> x_2>
\cdots > x_k$ of $ax_1+ax_2 + \cdots +ax_k \equiv b$ (mod $n$), we
obtain expressions. More precisely,
we prove:\\

\begin{thm}\label{strict ord thm}
Let $n$ be a positive integer and $b \in \mathbb{Z}_{n}$. Then, for
any given integer $a$ with $f=(a, n)$ and $f \mid b$, we have
\begin{equation*}
N_n(k, a, b)= \frac{(-1)^{k}f}{n}\sum_{d \mid (\frac{n}{f},k)}
(-1)^{\frac{k}{d}}  \binom{\frac{n}{d}}{\frac{k}{d}}C_{d}(b),
\end{equation*}
\end{thm}

\vskip 3mm

\noindent The special case when $a=1$ is due to K G Ramanathan (see,
Theorem 4, \cite{KGR44}). Of course, in the special case $(a,n)=1$,
Theorem \ref{strict ord thm} reduces to Ramanathan's theorem with $a^{-1}b$ in place of $b$, using a basic property of Ramanujan sums. Another special case
of Theorem \ref{strict ord thm} partially answers another question (Problem,
\cite{BKS19}) posed by Bibak, Kapron and Srinivasan; this is
described at the end of this section. \\

\noindent Surprisingly, Bibak does not seem to be aware of
Ramanathan's result. We also became aware of it only after we proved
Theorem \ref{strict ord thm} and found that the specific case where $a_{i}=1$ for all $i$ had already been handled in this almost 80-year old paper.
\vspace{2mm}

\noindent The general case of order-restricted solutions $x_1 \geq
\cdots \geq x_k$ of \eqref{linear cong} posed by Bibak is still
open, but we obtain the following partial result. Let $k=k_1+k_2+
\cdots + k_t$ be a partition of $k$, and we consider \eqref{linear
cong} where
\begin{equation}\label{k partition}
a_{1}=a_{2}= \cdots = a_{k_1}, a_{k_1+1}= \cdots = a_{k_1+k_2},
\cdots a_{k_1+ \cdots +k_{t-1}+1}= \cdots = a_{k_1+ \cdots + k_t}
\end{equation}
modulo $n$. Let $M_n(k_1, \dots, k_t, a_1, \dots, a_t, b)$ be the
number of solutions of the congruence
$$\sum_{i=1}^k a_ix_i \equiv
b \Mod{n}$$ satisfying $x_1 \geq x_2 \geq \cdots \geq x_{k_1}$,
$x_{k_1+1} \geq \cdots \geq x_{k_1+k_2}$, $\cdots$, $x_{k_1+ \cdots
+k_{t-1}+1} \geq \cdots \geq x_{k_1+ \cdots +k_t}$.
We prove:\\

\begin{thm}\label{perm thm}
Let $k=k_1+ \cdots + k_t$ and $a_1, \ldots, a_k$ be integers modulo
$n$, as above. Then
\begin{align*}
M_n(k_1, \dots, k_t, a_1, \dots, a_t, b)=& \frac{1}{n} \sum_{d_{1} \mid n} \dots \sum_{d_{t} \mid n} \frac{d_{1}}{d_{1}+\frac{k_{1}d_{1}}{n}} \cdots \frac{d_{t}}{d_{t}+\frac{k_{t}d_{t}}{n}}\\
& \binom{d_{1}+\frac{k_{1}d_{1}}{n}}{\frac{k_{1}d_{1}}{n}} \cdots \binom{d_{t}+\frac{k_{t}d_{t}}{n}}{\frac{k_{t}d_{t}}{n}} \sum_{\substack{m=1 \\ (a_{i}m, n)=d_{i} \\ i=1, \dots, t}}^{n} e\bigg(\frac{-bm}{n}\bigg).
\end{align*}
Suppose that $(a_i,n)=f$ for all $i$, then we have
\begin{equation*}
M_n(k_1, \dots, k_t, a_1, \dots, a_t, b)= \frac{f}{n}\sum_{d \mid (\frac{n}{f}, k_1, \dots,
k_t)} \frac{n^{t}}{(n+k_{1}) \cdots (n+k_{t})}
\binom{\frac{n+k_{1}}{d}}{\frac{k_{1}}{d}} \cdots
\binom{\frac{n+k_{t}}{d}}{\frac{k_{t}}{d}} C_{d}(b),
\end{equation*}
where $C_{d}(b)$ denotes the Ramanujan sum.
\end{thm}

\noindent Here, some of the $a_i$'s ($i \leq t$) may be equal; note,
for instance, that if $a_1=a_2$, the solutions counted by the
theorem correspond to separate orderings $x_1 \geq \cdots \geq
x_{k_1}$ and $x_{k_1+1} \geq \cdots \geq x_{k_1+k_2}$. The special
case of the first statement when $t=1$ and $a_i=1$ for all $i$, is
due to Bibak (\cite{KB20}). Moreover, when $k_1 = \dots = k_t = 1$,
the formula for $M_n(1, \dots, 1, a_1, \dots, a_k, b)$ simplifies to
$fn^{k-1}$ which aligns with Lehmer's theorem \cite{DNL13}, as it
accounts for all possible solutions of \eqref{linear cong} when
$a_1, \dots, a_k$ are distinct.\\

\noindent Theorem \ref{strict ord thm} can be interpreted as a subset sum problem in
the abelian group $\mathbb{Z}_n$. More generally, let $A$ be any
abelian group and let $D$ be a finite subset of $A$ containing $n$
elements. For a positive integer $1 \leq k \leq n$ and an element $b
\in A$, let $N_{D}(k, b)$ denote the number of $k$-element subsets
$S \subseteq D$ such that $\sum_{a \in S} a = b$. In particular,
when $A$ is a finite cyclic group $\mathbb{Z}_n$, the formula for
$N_D(k, b)$ is the same as the formula for $N_n(k, 1, b)$ discussed
earlier. In fact, for any finite abelian group $A$, Li and Wan
\cite{JW13} obtained an explicit formula for $N_D(k, b)$ when $D=A$.
Also, an analogus problem has been investigated in the context of
the finite field $\mathbb{F}_{q}$ of characteristic $p$ (see,
\cite{JW08}).

\vspace{1mm} \noindent The decision version of the subset sum
problem over $D$ is to determine $N_{D}(k, b) > 0$ for some $1  \leq
k \leq  n$. This problem has significant applications in coding
theory and cryptography. It is a well-known NP-complete problem,
even when $A$ is cyclic (finite or infinite) or the additive group
of a finite field $\mathbb{F}_q$. In particular, when
$A=\mathbb{Z}$, the subset sum problem forms the basis of the
knapsack cryptosystem. The case $A =\mathbb{F}_{q}$ is related to
the deep hole problem of extended Reed-Solomon codes, see
\cite{QCEM07}.

\subsection{Distinct Solutions and Sch\"{o}nemann's Theorem}

\noindent Another interesting problem related to the linear
congruences that has been considered in the literature is counting
distinct solutions. This problem was first addressed in a special
case by Sch\"{o}nemann \cite{TS1839} almost two centuries ago. Let
$D_{n}(b; a_{1}, \dots, a_{k})$ denote the number of solutions of
the linear congruence $a_{1}x_{1}+ \dots+ a_{k}x_{k} \equiv b \Mod
n,$ with all $x_{i}$ distinct. Sch\"{o}nemann \cite{TS1839} proved
the following result.

\vspace{1mm} \noindent \textbf{Theorem (Sch\"{o}nemann - 1839). }Let
$p$ be a prime, $a_{1}, \dots, a_{k}$ be arbitrary integers, and
$\sum_{i=1}^{k}a_{i} \equiv 0 \Mod p$ and $\sum_{i \in I}a_{i}
\not\equiv 0 \Mod p$ for all $\phi \neq I \subsetneq \{1, \dots,
k\}$. Then, the number $D_{p}(k, 0)$ is independent of the
coefficients $a_{1}, \dots, a_{k}$ and is equal to
\begin{equation*}
D_{p}(0; a_{1}, \dots, a_{k})= (-1)^{k-1}(k-1)! (p-1)+ (p-1) \cdots (p-k+1).
\end{equation*}

\vspace{1mm} \noindent Recently, in 2013, Grynkiewicz et al.
\cite{GPP13} obtained necessary and sufficient condition to
determine $D_{n}(b; a_{1}, \dots, a_{k}) >0$. In 2019, Bibak et al.
(\cite{BKS19}) generalized Sch\"{o}nemann's theorem using a graph
theoretic method. They proved the following result:
\begin{thm}(Theorem 2.3, \cite{BKS19})\label{dist soln thm}
Let $a_{1}, \dots, a_{k},b$ be arbitrary integers and $n$ be a positive integer, and  $(\sum_{i \in I}a_{i}, n)=1$ for all $\phi \neq I \subsetneq \{1, \dots, k\}$. Then, we have\\
$D_{n}(b; a_{1}, \dots, a_{k})$
\begin{equation*}
=\begin{cases}
    (-1)^{k}(k-1)!+ (n-1) \cdots (n-k+1), & \text{if } \big(\sum_{i=1}^{k}a_{i}, n\big)\nmid b,\\
     (-1)^{k-1}(k-1)! \big(\big(\sum_{i=1}^{k}a_{i}, n\big)-1\big)+ (n-1) \cdots (n-k+1), & \text{if } \big(\sum_{i= 1}^{k}a_{i}, n\big)\mid b.
    \end{cases}
\end{equation*}
\end{thm}

\noindent Furthermore, they posed a question about finding an
explicit formula for $D_{n}(b; a_{1}, \dots, a_{k})$ without
imposing any restrictions on the gcd of $a_i$'s and $n$ (see,
problem 1, \cite{BKS19}). When $a_{1}= \dots=a_{k}=a$, the problem
of counting solutions of strict order-restricted linear congruence
is equivalent to counting the distinct solutions of the linear
congruence up to permutations. As a consequence of Theorem
\ref{strict ord thm}, we have the following corollary:
\begin{corollary}\label{to strict ord thm}
Let $n$ be a positive integer and $b \in \mathbb{Z}_{n}$. Then, for any given integer $a$ with $f=(a, n)$ and $f \mid b$, we have
\begin{equation}\label{dist solu}
D_{n}(a, b)=D_{n}(b; a, \dots, a)=\frac{k!f(-1)^{k}}{n} \sum_{d \mid (\frac{n}{f},k)} (-1)^{\frac{k}{d}} \binom{\frac{n}{d}}{\frac{k}{d}}C_{d}(b),
\end{equation}
where $C_{d}(b)$ denote the Ramanujan sum.
\end{corollary}
 Since all the coefficients are $a$, the gcd conditions mentioned in the Theorem \ref{dist soln thm} are the same as to the conditions $(a\cdot i, n)=1$, for all $1 \leq i \leq k-1$. Thus, assuming these conditions on $n$ implies that $f=(a, n)=1$ and either $(k, n)=1$ or $k$ is the smallest prime divisor of $n$. Suppose $k$ is the smallest prime divisor of $n$, then by \eqref{dist solu}, we have
\begin{align*}
D_{n}(a, b)=&\frac{k!(-1)^{k}}{n}  \bigg((-1)^{k} \binom{n}{k}C_{1}(b)- \binom{\frac{n}{k}}{1}C_{k}(b)\bigg),\\
&=\begin{cases}
    (-1)^{k}(k-1)!+ (n-1) \cdots (n-k+1), & \text{if } k \nmid b,\\
     (-1)^{k-1}(k-1)! (k-1)+ (n-1) \cdots (n-k+1), & \text{if } k \mid b.
    \end{cases}
\end{align*}
Suppose $(k,n)=1$, then from \eqref{dist solu} it follows that
$$D_{n}(b)=(n-1) \cdots (n-k+1).$$
Therefore, the obtained formula \eqref{dist solu} provides a proof
of Theorem \ref{dist soln thm} in the special case when all
coefficients $a_1= \dots =a_k=a$. Moreover, this formula also
addresses the problem posed in \cite{BKS19} for the specific case
where $a_1, \dots, a_k$ are all equal to $a$. \vskip 5mm

\section{Square solutions - some subtleties}\label{sqar solns-subt}

\noindent We say that $(x_{1}, \dots, x_{k})$ is a square solution
of \eqref{linear cong}, if $x_{i}$'s are squares modulo $n$, for $1
\leq i \leq k$. Given integers $a_{1}, \dots, a_{k}, b$ and a
positive integer $n$, it is an interesting problem to determine the
necessary and sufficient conditions that guarantee the existence of
a square solution to the congruence represented by \eqref{linear
cong}. Thus, we are looking at the set
$$\{(x_1, \ldots, x_k) \in \mathbb{Z}_n^k : \sum_{i=1}^k a_ix_i
\equiv b \Mod{n}, x_i = y_i^2, y_i \in \mathbb{Z}_n \}.$$ We point
out that square solutions cannot simply be counted by enumerating
the set
$$\{(y_1, \ldots, y_k) \in \mathbb{Z}_n^k: \sum_{i=1}^k a_iy_i^2
\equiv b \Mod{n} \}.$$ For instance, if we consider the congruence
$x_1+x_2 \equiv 1 \Mod{27}$, there are four square solutions
$$\{(1,0), (0,1), (9,19), (19,9) \}$$
whereas, in the latter set
$$\{(y_1,y_2) \in \mathbb{Z}_9^2: y_1^2+y_2^2 \equiv 1 \Mod{27} \},$$
$6$ elements correspond to $(1,0)$ and $12$ elements correspond to
$(9,19)$. Thus, the complexity of this problem can vary
significantly, depending on the specific values of $a_i$, $b$, and
$n$. However, a necessary condition for the latter set to be
non-empty also gives a necessary condition for a square solution to
exist. One can easily see that when $n=p^{\ell}$, if $(b, p)=1$ and
there exists a subset $S$ of $\{a_{1}, \dots, a_{k}\}$ such that
$$s=\sum_{i \in S}a_{i} \in \mathbb{Z}_{p^{\ell}}^{*} \ \text{and }  \bigg(\frac{s}{p}\bigg)=\bigg(\frac{b}{p}\bigg)$$
then \eqref{linear cong} has a square solution $(x_{1}, \dots, x_{k})$, namely
\begin{equation*}
x_{i}= \begin{cases}
    s^{-1}b, & \text{if } a_{i} \in S,\\
     0, & \text{otherwise }.
    \end{cases}
\end{equation*}
But then, due to Lemma \ref{multi fn lem}, the above observation holds for any odd positive $n$. A related problem explored in the literature is the counting of representations of an integer $b$ as a sum of squares modulo $n$, where $b$ is any given integer. More generally, for given integers $a_1, \dots, a_k, b$ and positive integers $t_{1}, \dots, t_{k}$ and $n$ counting the number of solution of the congruence
\begin{equation}\label{gen cong}
a_{1}x_{1}^{t_{1}}+\dots+a_{k}x_{k}^{t_{k}} \equiv b \Mod n
\end{equation}
has also been studied in the literature; for example (see, \cite{CJAT15, JA19, HR32, SSY18, LT14} and \S 8.6 of \cite{IRBook}). When $n$ is an odd prime, $t$ is any positive integer and $t_{1}=\dots=t_{k}=t$, the solutions of \eqref{gen cong} were first studied by Lebesgue in 1837 (see, Chapter X of \cite{LED05}). In 1932, Hull \cite{HR32} proved a formula for counting the solutions to \eqref{gen cong}, when $t$ is any positive integer such that $t_{1}=\dots=t_{k}=t$ and $a_{1}= \dots =a_{k}=1$. Recently, T{\'o}th \cite{LT14} investigated the solutions of \eqref{gen cong} under various conditions on the $a_{i}$'s, depending on the modulus $n$ when $t_{1}=\dots=t_{k}=2$. More recently, Li and Ouyang \cite{SSY18} have provided an algorithm for computing the number of solutions of \eqref{gen cong} under the additional restriction that $x_i$ is a unit for every $i \in J \subseteq \{1, \dots, k\}$.

\vspace{1mm}
\noindent
Along with the other results, Hull (in Theorem 23, \cite{HR32}) discussed a sufficient condition on $k$ in order to \eqref{gen cong} has a solution, when $t \geq 2$ is a fixed integer, $t_{1}=\dots=t_{k}=t$ and $a_{1}= \dots =a_{k}=1$. Specifically, for the case when $t=2$ and $n=p^{\ell}$, where $p$ is an odd prime, the provided sufficient condition is as follows:
\begin{enumerate}[label=(\roman*)]
\item If $(b, p)=1$, then there is a solution whenever $k \geq 2$.
\item If $p \mid b$, then there is a solution whenever $k \geq 3$.
\end{enumerate}

\noindent
Notice that this sufficient condition also guarantees the existence of a square solution to the congruence represented by \eqref{linear cong}. This is because $x_{1}^{2}+\dots+x_{k}^{2} \equiv b \Mod n$ is equivalent to $y_{1}+ \dots+ y_{n} \equiv b \Mod n$, where $y_{i} \equiv x_{i}^{2} \Mod n$. As a result, \eqref{linear cong} has a square solution whenever $k \geq 3$ and $a_{1}= \dots =a_{k}=1$.

\vspace{1mm}
\noindent
Furthermore, in the case of $k=1$, the congruence \eqref{linear cong} has a square solution if and only if $b$ is a square. Also, for $k=2$, there exist linear congruences with $a_{1}=a_{2}=1$ and $(b, n)>1$ that do not have square solutions. For example, consider the congruence $x_{1}+x_{2} \equiv 3 \Mod 9$, which lacks square solutions as the squares modulo $9$ are ${0, 1, 4, 7}$.

\vspace{1mm}
Unlike the case of existence, the question of counting the number of square solutions for \eqref{linear cong} is not equivalent to the question of counting the number of representations as a sum of squares. This happens because the number of solutions to all solvable congruences of the form $x^{2} \equiv a \Mod n$ may not be the same for every integer $a$, although it is the same for every coprime $a$. As a result, determining the number of representations that correspond to the same square solution is rather difficult.

\vspace{1mm}
\noindent
For example, consider the square solutions of the congruence $x_{1}+x_{2} \equiv 1 \Mod{27}$, which are given by ${(1,0), (0,1), (19, 9), (9, 19)}$. In this case, the number of representations corresponding to $(1,0)$ is $6$, while the number of representations corresponding to $(19, 9)$ is $12$. This discrepancy arises because the congruence $x^{2} \equiv 0 \Mod{27}$ has $3$ solutions, and $x^{2} \equiv 9 \Mod{27}$ has $6$ solutions.

\vspace{1mm} \noindent Hence, our aim is to count the number of
square solutions for \eqref{linear cong}. Let $S_{n}(b; a_{1},
\dots, a_{k})$ denote the number of square solutions of
\eqref{linear cong}.

\noindent
Theorem \ref{square thm} provides a formula for $S_{n}(b; a_{1}, \dots, a_{k})$.
The proof of the theorem and its corollary stated in the
introduction are given in section \ref{proof of thm1}.

%%%%%%%%%%%%%%%%%%%%%%%%%%%%%%%%%%%%%%%%%%%%%%%%%%%%%%%%%%%%%%%%%%%%
\section{Preliminaries}\label{prelim}
\subsection{Discrete Fourier transform}
 An arithmetic function $f: \mathbb{Z} \rightarrow \mathbb{C}$ is said to be periodic with period $n$ (or $n$-periodic) for some $n \in \mathbb{N}$ if for every $b \in \mathbb{Z}$, $f(b+n)=f(b)$. From the definition \eqref{Rama sum}, it is clear that $C_n(b)$ is a periodic function of $b$ with period $n$. For an $n$-periodic arithmetic function $f$, its discrete (ﬁnite) Fourier transform (DFT) is deﬁned to be the function
\begin{equation*}
\hat{f}(b)= \sum_{j=1}^{n} f(j) e\bigg(\frac{-bj}{n}\bigg), \  \ \text{for} \
 \ b \in \mathbb{Z}.
\end{equation*}

\noindent
A Fourier representation of $f$ is given by
\begin{equation*}
f(b)= \frac{1}{n}\sum_{j=1}^{n} \hat{f}(j) e\bigg(\frac{bj}{n}\bigg), \  \ \text{for} \
 \ b \in \mathbb{Z},
\end{equation*}
which is the inverse discrete Fourier transform.
\subsection{Gauss sums and Ramanujan sums}
Let $n$ be a positive integer. A Dirichlet character $\chi$ modulo $n$ is an arithmetic function $\chi: \mathbb{Z} \rightarrow \mathbb{C}$ with period $n$ which is an extension of a group homomorphism from the multiplicative group $(\mathbb{Z}/n\mathbb{Z})^{*}$ to the set of complex numbers $\mathbb{C}$ via
\begin{equation*}
\chi(m)= \begin{cases}
    \chi(m \Mod n), & \text{if } (m,n)=1,\\
     0, & \text{if } (m,n) \geq 1.
    \end{cases}
\end{equation*}
Indeed, the extension of the trivial homomorphism $\chi_{0}$ is known as the principal character modulo $n$. This particular Dirichlet character, denoted as $\chi_0$, is defined as follows:
\begin{equation*}
\chi_{0}(m)= \begin{cases}
    1, & \text{if } (m,n)=1,\\
     0, & \text{if } (m,n) \geq 1.
    \end{cases}
\end{equation*}
The conductor of a Dirichlet character $\chi$ modulo $n$ is the smallest divisor of $n$ for which $\chi$ is periodic. We say that a Dirichlet character $\chi$ modulo $n$ is primitive if the conductor of $\chi$ is $n$. Otherwise, we say that $\chi$ is imprimitive.

\iffalse
For an odd prime $p$, the Legendre symbol $\big(\frac{a}{p}\big)$ is defined by
\begin{equation*}
\bigg(\frac{a}{p}\bigg)=
\begin{cases}
1, \quad \quad \text{if} \ x^{2} \equiv a \Mod{p} \ \text{has a solution} \ ,\\
-1, \quad \  \text{if} \ x^{2} \equiv a \Mod{p} \ \text{has no solutions},\\
0, \quad  \quad \text{if} \ p \mid a.
\end{cases}
\end{equation*}
Viewed as a function of $a$ (for fixed $p$), the Legendre symbol is a real primitive character modulo $p$.
\fi

\vspace{2mm}
For any Dirichlet character $\chi(m)$ to the modulus $n$, the Gauss sum
$\tau(\chi)$ is defined by
\begin{equation*}
\tau(\chi)= \sum_{m=1}^{n} \chi(m) e\bigg(\frac{m}{n}\bigg),
\end{equation*}
where $e(x)$ denote $e^{2\pi ix}$. Further, the more general Gauss sum is the discrete Fourier transform of a Dirichlet character $\chi$ modulo $n$, namely
\begin{equation*}
\tau_{b}(\chi)=\hat{\chi}(-b)= \sum_{m=1}^{n} \chi(m) e\bigg(\frac{bm}{n}\bigg), \  \ \text{for} \ \ b \in \mathbb{Z}.
\end{equation*}
For a Dirichlet character $\chi$ modulo $n$ induced by a primitive character $\chi^{\star}$ modulo $n^{\star}$, the following lemma reduces the computation of the general Gauss sum $\tau_{b}(\chi)$ to the Gauss sum $\tau(\chi^{\star})$.
\begin{lemma}(Theorem 9.7 and Theorem 9.12, \cite{MV07})\label{gen sum lem}
Let $\chi$ modulo $n$ be a non-principal character induced by the primitive character $\chi^{\star}$ modulo $n^{\star}$. Put $r=\frac{n}{(m, n^\star)}$. If $n^{\star} \nmid r$ then $\tau_{m}(\chi)=0$, while $n^{\star} \mid r$ then
\begin{equation*}
\tau_{m}(\chi)= \bar{\chi}^{\star}\bigg(\frac{m}{(n,m)}\bigg)\mu(r/n^{\star}) \chi^{\star}(r/n^{\star}) \frac{\varphi(n)}{\varphi(r)} \tau(\chi^{\star}).
\end{equation*}
In particular,
\begin{equation*}
\tau_{m}(\chi)=\bar{\chi}(m) \mu(n/n^{\star}) \chi^{\star}(n/n^{\star}) \tau(\chi^{\star}), \ \text{if} \ (m, n)=1.
\end{equation*}
Furthermore, we have
\[
|\tau_{m}(\chi^{\star})|=\sqrt{n^{\star}}.
\]
\end{lemma}
Though, we know that $|\tau(\chi)|=\sqrt{n}$ holds for any primitive character $\chi$ modulo $n$ the determination of the argument of the $|\tau(\chi)|$ is a difficult problem. In the case of real primitive characters $\tau(\chi)$ were evaluated completely by Gauss.
\begin{lemma}(Gauss)\label{Gauss lem}
Let $n \geq 1$ be an odd squarefree integer and let $\chi$ be a real primitive character modulo $n$. Then, we have
\[
\tau(\chi)=\varepsilon_{n}\sqrt{n},
\]
where $\varepsilon_{n}$ is defined as in \eqref{epsilon-p defn}.
\end{lemma}

When $\chi_{0}$ is the principal character modulo $n$, the general Gauss sum $\tau_{b}(\chi_{0})$ is called the Ramanujan sum, (i.e.)
\begin{equation}\label{Rama sum}
    \tau_{b}(\chi_{0})=C_{n}(b)= \sum_{\substack{j=1 \\ (j, n)=1}}^{n}e\bigg(\frac{jb}{n}\bigg).
\end{equation}
Now, we list some properties of the Ramanujan sums:
\begin{enumerate}[label=(\roman*)]
\item $C_{n}(b)$ is integer-valued.
\item For fixed $b \in \mathbb{Z}$ the function $b \rightarrow C_{n}(b)$ is multiplicative, (i. e.) if $(n_{1}, n_{2})=1$, then $C_{n_{1}n_{2}}(b)=C_{n_{1}}(b)C_{n_{2}}(b)$.
\item The function $b \rightarrow C_{n}(b)$ is multiplicative for a fixed $n$ if and only if $\mu(n)=1$, where $\mu$ denote the M\"{o}bius function.
\item $C_{n}(b)$ is an even function of $b$, that is, $C_{n}(b)=C_{n}((b,n))$, for every $b, n$. \label{Rama even prop}
\iffalse
\item Ramanujan sums satisfy the following orthogonality property. If $n \geq 1$, $d_{1} \mid n$, and $d_{2} \mid n$, then we have
\begin{equation*}
    \sum_{d \mid n} C_{d_{1}}\bigg(\frac{n}{d}\bigg) C_{d}\bigg(\frac{n}{d_{2}}\bigg) = \begin{cases}
    n, & \text{if } d_{1}=d_{2} ,\\
     0, & \text{if } d_{1} \neq d_{2}.
    \end{cases}
\end{equation*}
\fi
\item For integers $b$ and $n \geq 1$, we have
\begin{equation*}
    C_{n}(b)=\frac{\varphi(n)}{\varphi{\bigg(\frac{n}{(b, n)}\bigg)}} \mu\bigg(\frac{n}{(b, n)}\bigg).
\end{equation*}\label{Rama gen formula}
\end{enumerate}

\iffalse
Recall that an arithmetic function $f$ is $r$-even if $f(m)=f((m,r))$, for every $m \in \mathbb{Z}$. Clearly, if a function $f$ is $r$-even then it is $r$-periodic. Further, for an $r$-even function $f$ , we have
\begin{align*}
\hat{f}(b)=& \sum_{k=1}^{r} f(k) e\bigg(\frac{-bk}{r}\bigg)= \sum_{d \mid r} \sum_{\substack{1 \leq j \leq \frac{r}{d} \\ (j, \frac{r}{d})=1}}f(d) e\bigg(\frac{-bdj}{r}\bigg)= \sum_{d \mid r} f(d) C_{\frac{r}{d}}(b).
\end{align*}

The Cauchy convolution of two $r$-periodic functions $f$ and $g$ is deﬁned to be
$$(f \otimes g)(n)= \sum_{\substack{1 \leq x, y \leq m \\ x+y \equiv n \Mod{m}}} f(x) g(y)=\sum_{x=1}^{m} f(x) g(n-x), \ \text{for} \ m \in \mathbb{Z}.$$
Similarly, we can define the Cauchy convolution of a finite number of functions $r$-periodic functions. It is easy to observe that the discrete Fourier transform of the Cauchy convolution satisfies the relation $$\widehat{f \otimes g} = \hat{f} \hat{g},$$ with pointwise multiplication.
\fi
\subsection{Generating functions of partition with certain conditions}
For a positive integer $n$, a partition of $n$ is a non-increasing sequence of positive integers $p_{1}, p_{2}, \dots, p_{k}$ whose sum is $n$. Each $p_{i}$ is called a part of the partition. Let the function $p(n)$ denote the number of partitions of the integer $n$. It is well-known that the generating function of the sequence $\{p(n)\}_{n=0}^{\infty}$ is
\begin{align}\label{par gen fn}
\sum_{n\geq 0}p(n)q^{n}=&(1+q+q^{2}+q^{3}+\dots)(1+q+q^{2}+q^{3}+\dots)(1+q^{3}+q^{6}+\dots)\dots\nonumber\\
=&\frac{1}{1-q}\cdot \frac{1}{1-q^{2}}\cdot \frac{1}{1-q^{3}} \cdots = \prod_{i=1}^{\infty} \frac{1}{1-q^{i}} \ \ \text{for } \ |q| < 1.
\end{align}

\noindent
From a combinatorial perspective, the monomial chosen from the $i$-th parenthesis $1+q^i+q^{2i} +q^{3i}+\dots$ in \eqref{par gen fn} represents the number of times the part $i$ appears in the partition. In particular, if we choose the monomial $q^{n_{i}i}$ from the $i$-th parenthesis, then the value $i$ will appear $n_i$ times in the partition. Each selection of monomials makes one contribution to the coefficient of $q^n$ in the expression. More precisely, each contribution must be of the form $q^{1n_1} \cdot q^{2n_2} \cdot q^{3n_3} \cdots=q^{n_1 +2n_2+3n_3+\cdots}$. Thus, the coefficient of $q^n$ is the number of ways of writing $n=n_1 +2n_2+3n_3+\cdots$  where $n_1, n_2, n_3, \ldots$ are non-negative integers representing the total count of parts of size $1, 2, 3, \ldots$ in the partition, respectively.
% Hence, the generating function of a partition of $n$ with distinct parts (i.e.; when repetitions are not allowed) is given by
%\begin{align*}
 %\prod_{i=1}^{\infty} (1+q^{i}) \ \ \text{for } \ |q| < 1.
%\end{align*}

\vspace{1mm}
\noindent
Suppose $A=\{a_{1}, a_{2}, a_{3}, \dots \}$ is a set of positive integers. In general, the generating function for the number of partitions of $n$ into members of set $A$ is
\begin{align*}
f_{A}(q)=\prod_{a_{j} \in A} \frac{1}{1-q^{a_{j}}}.
\end{align*}
%when repetitions are allowed, and
%\begin{align*}
%g_{A}(q)=\prod_{a_{j} \in A} (1+q^{a_{j}})= \prod_{a_{j} \in A}^{\infty} \frac{1-q^{2a{j}}}{1-q^{a_{j}}},
%\end{align*}
%when repetitions are not allowed (see, for example, Theorem 1.1, \cite{GEA76} and \cite{HG70}).

%%%%%%%%%%%%%%%%%%%%%%%%%%%%%%%%%%%%%%%%%%%%%%%%%%%%%%%%%%%%%%%%%%%%%%%%%%%%%%
Moreover, the following lemma gives the generating function of the number of partitions of $n$ into $k$ parts each taken from the given set $A$.
\begin{lemma} \cite{HG70} \label{gen fn lem}
Let $A$ be a set of positive integers. Let $k$ be a positive integer and $b$ be a non-negative integer. The number of partitions of $b$ into $k$ parts each taken
from the set $A$, is the coefficient of $q^{b}z^{k}$ in
$$\prod_{a_{j} \in A} \frac{1}{1-zq^{a_{j}}}.$$
%when repetitions are allowed, and
%$$\prod_{a_{j} \in A} (1+zq^{a_{j}}),$$
%when repetitions are not allowed.
Furthermore, if we multiply the coefficient of $q^{b}z^{k}$ in the expression:
$$\prod_{a_{j} \in A} (1-zq^{a_{j}})$$
by $(-1)^{k}$, then we obtain the number of distinct partitions of $b$ into exactly $k$ parts, with each taken from the set $A$.
\end{lemma}
\vskip 5mm

\section{Square Solutions - proof of Theorem 1}\label{proof of thm1}

\noindent It is convenient to use the characteristic function for
squares in our proofs. Using this and Hensel's lemma, we prove
multiplicativity for square solutions. 

\vspace{2mm}
\noindent
Recall that an element $a \in \mathbb{Z}_{n}$ is a square in
$\mathbb{Z}_{n}$ (or square modulo $n$) if and only if $x^{2} \equiv
a \Mod n$ has a solution. The units (elements of $\mathbb{Z}_{n}$
that are relatively prime to $n$) that are squares are called
quadratic residues modulo $n$. We define a function $\square_{n}:
\mathbb{Z}_{n} \rightarrow \mathbb{Z}_{n}$ by
\begin{equation}\label{square char eq}
\square_{n}(b)= \begin{cases}
1 ,& \ \ \text{if } \  b \text{ is a square modulo} \ n ,\\
   0 ,& \ \ \text{otherwise}.
\end{cases}
\end{equation}
\vskip 3mm

\noindent The following statement is a version of Hensel's lemma.
\vskip 2mm

\begin{lemma}\label{Hensel lem}
Suppose $f(x) \in \mathbb{Z}[x]$ and $f(a) \equiv 0 \Mod{p^{m}}$ and
$f'(a) \not\equiv 0 \Mod p$. Then there is a unique $t \in \{0,1,
\dots, p-1\}$ such that $f(a+tp^{m}) \equiv 0 \Mod{p^{m+1}}$.
\end{lemma}

\vskip 3mm

\noindent We use the following observations in the next lemma.\\
Let $s(n)$ and $q(n)$ denote the number of squares in
$\mathbb{Z}_{n}$ and the number of quadratic residues in
$\mathbb{Z}_{n}$ respectively. Equivalently,
$$s(n)= \sum_{b=1}^{n} \square_{n}(b) \ \ \text{and } \ \ q(n)= \sum_{\substack{b=1 \\ (b,n)=1}}^{n} \square_{n}(b).$$
It is well known that $q(n)$ is a multiplicative function. Stangl
\cite{WDS96}, showed that $s(n)$ is a multiplicative function.
Furthermore, for any odd prime $p$, he derived the following
recursion formula for $s(p^{r})$:
\begin{equation}\label{square rec}
s(p^{r})=q(p^{r})+s(p^{r-2}), \ \ \ \text{for } r \geq 3
\end{equation}
and $s(p)=q(p)+1=\frac{p+1}{2}$,
$s(p^{2})=q(p^{2})+1=\frac{p^{2}-p+2}{2}$. This recursion formula
follows from the observation that an element $b$ is a square in
$\mathbb{Z}_{p^{r-2}}$ if and only if $bp^{2}$ is a square in
$\mathbb{Z}_{p^{r}}$. As a consequence of this recursion formula, we
have the following lemma. \vskip 3mm

\begin{lemma}\label{sq decom lem}
For an odd prime $p$ and a positive integer $\ell$, we have
\begin{equation}\label{sq decom eq}
\sum_{x=1}^{p^{\ell}}\square_{p^{\ell}}(x)e\bigg(\frac{xm}{p^{\ell}}\bigg)=1+\frac{1}{2}
\sum_{\substack{j=0 \\ j \equiv 0 \Mod 2}}^{\ell-1}
\sum_{\substack{x=1 \\
(x,p^{\ell-j})=1}}^{p^{\ell-j}}\bigg(1+\bigg(\frac{x}{p}\bigg)\bigg)e\bigg(\frac{xm}{p^{\ell-j}}\bigg),
\end{equation}
where $\left(\frac{\cdot}{p}\right)$ is the Legendre symbol mod $p$.
\end{lemma}
\begin{proof}
For any $x \in \mathbb{Z}_{p^\ell}$ with $(x, p^\ell) = 1$, we can
express $x$ as $$x = x_{\ell-1}p^{\ell-1} + \dots + x_{1}p + x_0,$$
where $0 \leq x_i \leq p-1$ for $0 \leq i \leq \ell-1$ and $x_0 \neq
0$. If $x_0$ is a residue modulo $p$, then the congruence $y^2
\equiv x \Mod p$ has a solution. Using Lemma \ref{Hensel lem} with
the polynomial $f(x) = y^2 - x$, we can lift this solution modulo
$p^\ell$. Alternatively, we can define
$\frac{1}{2}\left(1+\left(\frac{x}{p}\right)\right)$ as a
characteristic function for quadratic residues modulo $p^\ell$.

\vspace{1mm} \noindent We now prove \eqref{sq decom eq} by induction
on $\ell$. If $\ell=1$, then \eqref{sq decom eq} follows from the
definition of the Legendre symbol. If $\ell=2$, then we have
$s(p^{2})=q(p^{2})+1$. Therefore, \eqref{sq decom eq} follows from
the above observation.

\vspace{1mm}
\noindent
Now, we assume that the claim is true if $\ell \leq
m-1$. By using \eqref{square rec} and the hypothesis, we write
\begin{align*}
\sum_{x=1}^{p^{m}}\square_{p^{m}}(x)e\bigg(\frac{xm}{p^{m}}\bigg)=&1+\frac{1}{2} \sum_{\substack{j=0 \\ j \equiv 0 \Mod 2}}^{m-3} \sum_{\substack{x=1 \\ (x,p^{m-2-j})=1}}^{p^{m-2-j}}\bigg(1+\bigg(\frac{x}{p}\bigg)\bigg)e\bigg(\frac{xm}{p^{m-2-j}}\bigg)\\
&+\sum_{\substack{x=1 \\
(x,p^{m})=1}}^{p^{m}}\square_{p^{m}}(x)e\bigg(\frac{xm}{p^{m}}\bigg).
\end{align*}
By using the above observation, we write
\begin{align*}
\sum_{x=1}^{p^{m}}\square_{p^{m}}(x)e\bigg(\frac{xm}{p^{m}}\bigg)=&1+\frac{1}{2} \sum_{\substack{j=0 \\ j \equiv 0 \Mod 2}}^{m-3} \sum_{\substack{x=1 \\ (x,p^{m-2-j})=1}}^{p^{m-2-j}}\bigg(1+\bigg(\frac{x}{p}\bigg)\bigg)e\bigg(\frac{xm}{p^{m-2-j}}\bigg)\\
&+\frac{1}{2}\sum_{\substack{x=1 \\ (x,p^{m})=1}}^{p^{m}}\bigg(1+\bigg(\frac{x}{p}\bigg)\bigg)e\bigg(\frac{xm}{p^{m}}\bigg)\\
=&1+\frac{1}{2} \sum_{\substack{j=0 \\ j \equiv 0 \Mod 2}}^{m-1}
\sum_{\substack{x=1 \\
(x,p^{m-j})=1}}^{p^{m-j}}\bigg(1+\bigg(\frac{x}{p}\bigg)\bigg)e\bigg(\frac{xm}{p^{m-j}}\bigg).
\end{align*}
This completes the proof of Lemma \ref{sq decom lem}.
\end{proof}

\vskip 3mm

\noindent
 Recall that $S_{n}(b; a_{1}, \dots, a_{k})$ denotes the number of square solutions of \eqref{linear cong}.
Now, we show that the function $n \rightarrow S_{n}(b; a_{1}, \dots,
a_{k})$ is multiplicative, for any given integers $a_{1}, \dots,
a_{k}, b$. \vskip 3mm

\begin{lemma}\label{multi fn lem}
Let $a_{1}, \dots, a_{k}, b$ be integers. Then, for any $n$ and $n'$
relatively prime, we have
$$S_{nn'}(b; a_{1}, \dots, a_{k})= S_{n}(b; a_{1}, \dots, a_{k}) \cdot S_{n'}(b; a_{1}, \dots, a_{k}).$$
\end{lemma}
\begin{proof}
It is easy to see that every square solution of the linear
congruence $a_{1}x_{1}+\dots+a_{k}x_{k} \equiv b \Mod{nn'}$
corresponds a square solution of $a_{1}x_{1}+\dots+a_{k}x_{k} \equiv
b \Mod{n}$ and a square solution of $a_{1}x_{1}+\dots+a_{k}x_{k}
\equiv b \Mod{n'}$. Thus, we have
$$S_{nn'}(b; a_{1}, \dots, a_{k}) \leq S_{n}(b; a_{1}, \dots, a_{k}) \cdot S_{n'}(b; a_{1}, \dots, a_{k}).$$
Conversely, let $(x_{1}, \dots, x_{k})$ and $(x'_{1}, \dots,
x'_{k})$ be square solutions of $a_{1}x_{1}+\dots+a_{k}x_{k} \equiv
b \Mod{n}$ and $a_{1}x'_{1}+\dots+a_{k}x'_{k} \equiv b \Mod{n'}$
respectively. Since $(n, n')=1$, it follows from the Chinese
remainder theorem that there is a unique $\gamma_{i}$ modulo $nn'$
for each $1 \leq i \leq k$ such that
$$\gamma_{i} \equiv x_{i} \Mod n, \  \gamma_{i} \equiv x'_{i} \Mod{n'}, \ \text{for} \ 1 \leq i \leq k.$$
Therefore, we have
\begin{align*}
a_{1}\gamma_{1}+\dots+a_{k}\gamma_{k}-b &\equiv a_{1}x_{1}+\dots+a_{k}x_{k}-b \equiv 0 \Mod{n},\\
a_{1}\gamma_{1}+\dots+a_{k}\gamma_{k}-b &\equiv
a_{1}x'_{1}+\dots+a_{k}x'_{k}-b \equiv 0 \Mod{n'}
\end{align*}
and hence $(n, n')=1$ implies that
$a_{1}\gamma_{1}+\dots+a_{k}\gamma_{k} \equiv b \Mod{nn'}$. Also,
for each $1 \leq i \leq k$, the congruences
$$y_{i}^{2} \equiv \gamma_{i} \Mod n \equiv x_{i} \Mod n, \ \ y_{i}^{2} \equiv \gamma_{i} \Mod{n'} \equiv x'_{i} \Mod{n'},$$
has solutions implies that $n \mid (y_{i}^{2} - \gamma_{i})$ and $n'
\mid (y_{i}^{2} -\gamma_{i})$, for each $1 \leq i \leq k$. Since $n$
and $n'$ relatively prime, we obtain $y_{i}^{2} \equiv \gamma_{i}
\Mod{nn'}$. Therefore, $(\gamma_{1}, \dots, \gamma_{k})$ is a square
solution of $a_{1}x_{1}+\dots+a_{k}x_{k} \equiv b \Mod{nn'}$. This
completes the proof of Lemma \ref{multi fn lem}.
\end{proof}
\vskip 3mm

\noindent We need one final lemma on Gauss sums that will be used in
proving Theorem 3 on square solutions. \vskip 3mm

\begin{lemma}\label{Gauss sum est}
Let $p$ be an odd prime and $\ell$ be a positive integer. Let
$\chi_{p^{\ell}}$ be a real character modulo $p^{\ell}$ induced by
the Legendre symbol $\bigg(\frac{\cdot}{p}\bigg)$ modulo $p$. Then,
for any positive integer $m$, we have
\[
\sum_{x=1}^{p^{\ell}}\chi_{p^{\ell}}(x)
e\bigg(\frac{mx}{p^{\ell}}\bigg)= \begin{cases}
\varepsilon_{p} \bigg(\frac{m/p^{\ell-1}}{p}\bigg) p^{\ell-\frac{1}{2}}, & \ \ \text{if } \  (m, p^{\ell})=p^{\ell-1},\\
   0, & \ \ \text{otherwise}.
\end{cases}
\]
where $\varepsilon_{p}$ is defined as in \eqref{epsilon-p defn}.
\end{lemma}
\begin{proof}
If $p^{\ell} \mid m$, then it's straightforward to see that the required sum vanishes. Suppose $p^{\ell} \nmid m$, then by using Lemma \ref{gen sum lem} with $n=p^{\ell}, n^{\star}=p$ and
$r=\frac{p^{\ell}}{(p^{\ell},m)}$, we obtain
\[
\sum_{x=1}^{p^{\ell}}\chi_{p^{\ell}}(x)
e\bigg(\frac{mx}{p^{\ell}}\bigg)= \begin{cases}
\bigg(\frac{m/(p^{\ell}, m)}{p}\bigg)\mu(r/p) \bigg(\frac{r/p}{p}\bigg) \frac{\varphi(p^{\ell})}{\varphi(r)} \tau\bigg(\bigg(\frac{\cdot}{p}\bigg)\bigg), & \ \ \text{if } \  (m, p) \mid r,\\
   0, & \ \ \text{if } \  (m, p) \nmid r.
\end{cases}
\]
It follows from the definition of the Legendre symbol that
\[\bigg(\frac{r/p}{p}\bigg)=\bigg(\frac{p^{\ell-1}/(m,p^{\ell})}{p}\bigg)=\begin{cases}
1, & \ \ \text{if } \  (m, p^{\ell})=p^{\ell-1},\\
   0, & \ \ \text{otherwise},
\end{cases}
\]
Thus the sum is possibly nonzero only when $(m,
p^{\ell})=p^{\ell-1}$. Suppose $(m, p^{\ell})=p^{\ell-1}$, then we
have
$$\sum_{x=1}^{p^{\ell}}\chi_{p^{\ell}}(x) e\bigg(\frac{mx}{p^{\ell}}\bigg)=
\bigg(\frac{m/p^{\ell-1}}{p}\bigg)\mu(1) \bigg(\frac{1}{p}\bigg)
\frac{\varphi(p^{\ell})}{\varphi(p)}
\tau\bigg(\bigg(\frac{\cdot}{p}\bigg)\bigg)=p^{\ell-1}
\bigg(\frac{m/p^{\ell-1}}{p}\bigg)
\tau\bigg(\bigg(\frac{\cdot}{p}\bigg)\bigg).$$ Since the Legendre
symbol is a real primitive character modulo $p$, by using Lemma
\ref{Gauss lem}, we obtain the required estimate.
\end{proof}

\vskip 2mm
\noindent
\textbf{Proof of Theorem \ref{square thm}.} Our aim is to calculate $S_{n}(b; a_{1}, \dots, a_{k})$.
Since $n=p_{1}^{\ell_{1}}\cdots p_{r}^{\ell_{r}}$, by applying Lemma
\ref{multi fn lem}, it is enough to calculate
$S_{p_{i}^{\ell_{i}}}(b; a_{1}, \dots, a_{k})$, for $1 \leq i \leq
r$.  For any $p \in \{p_{1}, \dots, p_{r}\}$, we write
\begin{align*}
S_{p^{\ell}}(b; a_{1}, \dots, a_{k})&=\frac{1}{p^{\ell}}\sum_{m=1}^{p^{\ell}}\bigg(\sum_{x_{1}=1}^{p^{\ell}} \square_{p^{\ell}}(x_{1}) \cdots \sum_{x_{k}=1}^{p^{\ell}} \square_{p^{\ell}}(x_{k}) e\bigg(\frac{(a_{1}x_{1}+\dots+a_{k}x_{k}-b)m}{p^{\ell}}\bigg)\bigg)\\
&=\frac{1}{p^{\ell}}\sum_{m=1}^{p^{\ell}}e\bigg(\frac{-bm}{p^{\ell}}\bigg)\sum_{x_{1}=1}^{p^{\ell}}
\square_{p^{\ell}}(x_{1}) e\bigg(\frac{a_{1}x_{1}m}{p^{\ell}}\bigg)
\cdots \sum_{x_{k}=1}^{p^{\ell}}
\square_{p^{\ell}}(x_{k})e\bigg(\frac{a_{k}x_{k}m}{{p^{\ell}}}\bigg)
\end{align*}
where $\square_{p^{\ell}}$ is defined as in \eqref{square char eq}.
Then, by using Lemma \ref{sq decom lem}, we obtain
\begin{align*}
S_{p^{\ell}}(b; a_{1}, \dots, a_{k})=&\frac{1}{p^{\ell}} \sum_{m=1}^{p^{\ell}} e\bigg(\frac{-bm}{p^{\ell}}\bigg)\bigg(1+\frac{1}{2} \sum_{\substack{j_{1}=0 \\ j_{1} \equiv 0 \Mod 2}}^{\ell-1} \sum_{\substack{x_{1}=1 \\ (x_{1},p^{\ell-j_{1}})=1}}^{p^{\ell-j_{1}}}\bigg(1+\bigg(\frac{x_{1}}{p}\bigg)\bigg)e\bigg(\frac{a_{1}x_{1}m}{p^{\ell-j_{1}}}\bigg)\bigg) \\
& \cdots \bigg(1+\frac{1}{2} \sum_{\substack{j_{k}=0 \\ j_{k} \equiv 0 \Mod 2}}^{\ell-1} \sum_{\substack{x_{k}=1 \\ (x_{k},p^{\ell-j_{k}})=1}}^{p^{\ell-j_{k}}}\bigg(1+\bigg(\frac{x_{k}}{p}\bigg)\bigg)e\bigg(\frac{a_{k}x_{k}m}{{p^{\ell-j_{k}}}}\bigg) \bigg)\\
&=\frac{1}{p^{\ell}} \sum_{m=1}^{p^{\ell}}
e\bigg(\frac{-bm}{p^{\ell}}\bigg)J_{1}(m) \cdots J_{k}(m),
\end{align*}
where $$J_{i}(m)=1+\frac{1}{2} \sum_{\substack{j_{i}=0 \\ j_{i}
\equiv 0 \Mod 2}}^{\ell-1} \sum_{\substack{x_{i}=1 \\
(x_{i},p^{\ell-j_{i}})=1}}^{p^{\ell-j_{i}}}\bigg(1+\bigg(\frac{x_{i}}{p}\bigg)\bigg)e\bigg(\frac{a_{i}x_{i}m}{p^{\ell-j_{i}}}\bigg),
\ \text{for} \ 1 \leq i \leq k.$$ Now, consider
\begin{align*}
J_{i}(m)=&1+\frac{1}{2} \sum_{\substack{j_{i}=0 \\ j_{i} \equiv 0 \Mod 2}}^{\ell-1} \sum_{\substack{x_{i}=1 \\ (x_{i},p^{\ell-j_{i}})=1}}^{p^{\ell-j_{i}}}e\bigg(\frac{a_{i}x_{i}m}{p^{\ell}}\bigg)+\frac{1}{2} \sum_{\substack{j_{i}=0 \\ j_{i} \equiv 0 \Mod 2}}^{\ell-1} \sum_{\substack{x_{i}=1 \\ (x_{i},p^{\ell-j_{i}})=1}}^{p^{\ell-j_{i}}}\bigg(\frac{x_{i}}{p}\bigg)e\bigg(\frac{a_{i}x_{i}m}{p^{\ell-j_{i}}}\bigg)\\
=&1+\frac{1}{2} \sum_{\substack{j_{i}=0 \\ j_{i} \equiv 0 \Mod
2}}^{\ell-1} C_{p^{\ell-j_{i}}}(a_{i}m)+\frac{1}{2}
\sum_{\substack{j_{i}=0 \\ j_{i} \equiv 0 \Mod 2}}^{\ell-1}
\sum_{\substack{x_{i}=1 \\
(x_{i},p^{\ell-j_{i}})=1}}^{p^{\ell-j_{i}}}\bigg(\frac{x_{i}}{p}\bigg)e\bigg(\frac{a_{i}x_{i}m}{p^{\ell-j_{i}}}\bigg),
\end{align*}
where $C_{p^{\ell-j_{i}}}(a_{i}m)$ is a Ramanujan's sum defined as
in \eqref{Rama sum}. By using Lemma \ref{Gauss sum est}, we write
\begin{equation*}
J_{i}(m)= 1+\frac{1}{2} \sum_{\substack{j_{i}=0 \\ j_{i} \equiv 0
\Mod 2}}^{\ell-1} \bigg( C_{p^{\ell-j_{i}}}(a_{i}m)+\varepsilon_{p}
\bigg(\frac{a_{i}m/p^{\ell-j_{i}-1}}{p}\bigg)\varrho_{j_{i}}(a_{i}m)
p^{\ell-j_{i}-\frac{1}{2}}\bigg),
\end{equation*}
where $\varrho_{j_{i}}(x)$ is defined as in \eqref{coprime char fn}.
Therefore, we have
\begin{align*}
S_{p^{\ell}}(b; a_{1}, \dots, a_{k})=&\frac{1}{p^{\ell}} \sum_{m=1}^{p^{\ell}} e\bigg(\frac{-bm}{p^{\ell}}\bigg) \prod_{i=1}^{k} \bigg(1+\frac{1}{2} \sum_{\substack{j_{i}=0 \\ j_{i} \equiv 0 \Mod 2}}^{\ell-1} \bigg( C_{p^{\ell-j_{i}}}(a_{i}m)\\
&+\varepsilon_{p}\bigg(\frac{a_{i}m/p^{\ell-j_{i}-1}}{p}\bigg)\varrho_{j_{i}}(a_{i}m)
p^{\ell-j_{i}-\frac{1}{2}}\bigg)\bigg) =\frac{1}{p^{\ell}}
\bigg(\sum_{m=1}^{p^{\ell}}
e\bigg(\frac{-bm}{p^{\ell}}\bigg)+\sum_{\substack{K \subset \{1,
\dots, k\} \\ K \neq \phi}}\frac{1}{2^{|K|}}  S_{K}\bigg),
\end{align*}
where $S_{K}$ is defined as in \eqref{SK eq}. This completes the proof of Theorem \ref{square thm}.

\vskip 3mm
\noindent
\textbf{Proof of the the corollary \ref{square corollary}. }
Since $(a_{i}, n)=1$ for all $i$, \eqref{SK eq} can be written as
\begin{align*}
S_{K}&
=\sum_{r=0}^{\ell}\sum_{\substack{m=1 \\ (m,p^{\ell})=p^r}}^{p^{\ell}}e\bigg(\frac{-bm}{p^{\ell}}\bigg) \bigg(\sum_{\substack{j=0 \\ j \equiv 0 \Mod 2}}^{\ell-1} \bigg( C_{p^{\ell-j}}(m)+\varepsilon_{p} \bigg(\frac{m/p^{\ell-j-1}}{p}\bigg) \varrho_{j}(m) p^{\ell-j-\frac{1}{2}}\bigg)\bigg)^{|K|}\\
&= \bigg(\sum_{\substack{j=0 \\ j \equiv 0 \Mod 2}}^{\ell-1}  C_{p^{\ell-j}}(p^\ell)\bigg)^{|K|}+S'_{K},
\end{align*}
where
$$S'_{k}=\sum_{r=0}^{\ell-1}\sum_{\substack{m=1 \\ (m,p^{\ell})=p^r}}^{p^{\ell}}e\bigg(\frac{-bm}{p^{\ell}}\bigg) \bigg(\sum_{\substack{j=0 \\ j \equiv 0 \Mod 2}}^{\ell-1} \bigg( C_{p^{\ell-j}}(m)+\varepsilon_{p} \bigg(\frac{m/p^{\ell-j-1}}{p}\bigg) \varrho_{j}(m) p^{\ell-j-\frac{1}{2}}\bigg)\bigg)^{|K|}.$$

\noindent
As we assume $(b,n)=1$, by using the property \ref{Rama gen formula} of the Ramanujan sum and Lemma \ref{gen sum lem}, we obtain that the both sum
$$\sum_{\substack{m=1 \\ (m,p^{\ell})=p^r}}^{p^{\ell}}e\bigg(\frac{-bm}{p^{\ell}}\bigg) \ \ \text{and } \sum_{\substack{m=1 \\ (m,p^{\ell})=p^r}}^{p^{\ell}}e\bigg(\frac{-bm}{p^{\ell}}\bigg)\bigg(\frac{m}{p}\bigg).$$
vanishes for $r \leq \ell-2$. So, when we apply the binomial expansion, we notice that the inner sums of terms where  $r \leq \ell-2$ vanishes. Thus, using the definition of $\varrho_{j}$ (Definition \ref{coprime char fn}) and the property \ref{Rama gen formula} of the Ramanujan sum, we write
$$S'_{K}
=\sum_{m=1}^{p-1}e\bigg(\frac{-bm}{p}\bigg)\bigg(-p^{\ell-1}+\sum_{\substack{j=2 \\ j \equiv 0 \Mod 2}}^{\ell-1}\phi(p^{\ell-j})+\varepsilon_{p}\bigg(\frac{m}{p}\bigg)p^{\ell-\frac{1}{2}}\bigg)^{|K|}.$$
This completes the proof of Corollary \ref{square corollary}.

\noindent Note that if $k=2, a_1=a_2=1, b=2, p=3, n=p^2$, the
corollary above gives the value $p(p+1)/4 = 3$ which counts all the
square solutions
$$\{(1,1),(4,7),(7,4) \}$$
of $x_1 + x_2 \equiv 2$ (mod $9$). \vskip 5mm

\section{Order restricted congruence - proof of Theorem 3}\label{proof of thm3}

\noindent In this section, we discuss the proof of Theorem \ref{perm
thm}. We shall use the following lemma.
\begin{lemma}(Lemma IV.6, \cite{BKMO19})\label{exp prod lem}
Let $n$ be a positive integer and $a, m$ be non-negative integers. Then, we have
$$\prod_{j=1}^{n}\bigg(1-z e\bigg(\frac{jam}{n}\bigg)\bigg)=(1-z^{\frac{n}{d}})^{d},\ \text{where} \ d=(am,n).$$
\end{lemma}
\iffalse
\begin{proof}
As we assume that $(a, n) = 1$, it follows that $(am, n)=(m, n)=d$. Therefore, by applying Lemma IV.6 of \cite{BKMO19}, we obtain
\begin{equation*}
\prod_{j=1}^{n}\bigg(1-z^{a}e\bigg(\frac{jam}{n}\bigg)\bigg)=(1-z^{\frac{an}{(am,n)}})^{(am,n)}=(1-z^{\frac{an}{d}})^{d}.
\end{equation*}
\end{proof}
\fi

\vskip 5mm

\noindent \textbf{Proof of Theorem \ref{perm thm}. }Let $k = k_1 +
\dots + k_t$ be a partition of $k$ given as in \eqref{k partition}.
By using Lemma \ref{gen fn lem}, we see that the number of
partitions of $b$ into $k$ parts such that exactly $k_{1}$ parts
taken from the set $A_{1}$, exactly $k_{2}$ parts taken from the set
$A_{2}$, and so on, up to exactly $k_{2}$ parts taken from the set
$A_{2}$, is the coefficient of $q^{b}z_{1}^{k_{1}} \cdots
z_{t}^{k_{t}}$ in
\begin{equation*}
\prod_{j_{1} \in A_{1}} \frac{1}{1-z_{1}q^{j_{1}}} \times \cdots \times \prod_{j_{t} \in A_{t}} \frac{1}{1-z_{t}q^{j_{t}}}.
\end{equation*}
Taking $A_{i}=\{a_{i}, 2a_{i}, \dots, na_{i}\}$, for $i=1, \dots, t$
and $q=e^{2\pi im/n}$, where $m$ is a non-negative integer, we
observe
$$\sum_{b=1}^{n}M_n(k_1, \dots, k_t, a_1, \dots, a_t, b) e\bigg(\frac{bm}{n}\bigg)$$
to be the coefficient of $z_{1}^{k_{1}}\cdots z_{t}^{k_{t}}$ in
$$\prod_{i=1}^{t} \prod_{j_{i}=1}^{n} \bigg(1-z_{i}e\bigg(\frac{j_{i}m}{n}\bigg)\bigg)^{-1}.$$
By using Lemma \ref{exp prod lem}, we obtain
$$\sum_{b=1}^{n}M_n(k_1, \dots, k_t, a_1, \dots, a_t, b) e\bigg(\frac{bm}{n}\bigg)=(-1)^{\frac{k_{1}d_{1}+\dots+k_{t}d_{t}}{n}}\binom{-d_{1}}{\frac{k_{1}d_{1}}{n}} \cdots \binom{-d_{t}}{\frac{k_{t}d_{t}}{n}},$$
where $d_{i}=(a_{i}m,n)$, for $i=1, \dots, t$. Now, by taking the
inverse Fourier transform, we write $M_n(k_1, \dots, k_t, a_1,\dots,
a_t, b)$ as
\begin{align*}
& \frac{1}{n}\sum_{m=1}^{n} (-1)^{\frac{k_{1}d_{1}+\dots+k_{t}d_{t}}{n}}\binom{-d_{1}}{\frac{k_{1}d_{1}}{n}} \cdots \binom{-d_{t}}{\frac{k_{t}d_{t}}{n}} e\bigg(\frac{-bm}{n}\bigg)\\
=&\frac{1}{n}\sum_{m=1}^{n} \frac{d_{1}}{d_{1}+\frac{k_{1}d_{1}}{n}} \cdots \frac{d_{t}}{d_{t}+\frac{k_{t}d_{t}}{n}} \binom{d_{1}+\frac{k_{1}d_{1}}{n}}{\frac{k_{1}d_{1}}{n}} \cdots \binom{d_{t}+\frac{k_{t}d_{t}}{n}}{\frac{k_{t}d_{t}}{n}} e\bigg(\frac{-bm}{n}\bigg)\\
=&\frac{1}{n} \sum_{d_{1} \mid n} \dots \sum_{d_{t} \mid n}
\frac{d_{1}}{d_{1}+\frac{k_{1}d_{1}}{n}} \cdots
\frac{d_{t}}{d_{t}+\frac{k_{t}d_{t}}{n}}
\binom{d_{1}+\frac{k_{1}d_{1}}{n}}{\frac{k_{1}d_{1}}{n}} \cdots
\binom{d_{t}+\frac{k_{t}d_{t}}{n}}{\frac{k_{t}d_{t}}{n}}
\sum_{\substack{m=1 \\ (a_{i}m, n)=d_{i} \\ i=1, \dots, t}}^{n}
e\bigg(\frac{-bm}{n}\bigg).
\end{align*}
If we assume $(a_{i}, n)=f$ for $i=1, \dots, t$, we write
$d_{i}=(a_{i}m,n)=f(m, \frac{n}{f})=fd$, for $i=1, \dots, t$. Thus,
we have
\begin{align*}
&\frac{f}{n}\sum_{d \mid \frac{n}{f}} \frac{df}{df+\frac{k_{1}df}{n}} \cdots \frac{df}{df+\frac{k_{t}df}{n}} \binom{df+\frac{k_{1}df}{n}}{\frac{k_{1}df}{n}} \cdots \binom{df+\frac{k_{t}df}{n}}{\frac{k_{t}df}{n}} C_{\frac{n}{df}}(-b)\\
=& \frac{f}{n}\sum_{d \mid \frac{n}{f}} \frac{\frac{n}{d}}{\frac{n}{d}+\frac{k_{1}}{d}} \cdots \frac{\frac{n}{d}}{\frac{n}{d}+\frac{k_{t}}{d}} \binom{\frac{n}{d}+\frac{k_{1}}{d}}{\frac{k_{1}}{d}} \cdots \binom{\frac{n}{d}+\frac{k_{t}}{d}}{\frac{k_{t}}{d}} C_{d}(b)\\
=&\frac{f}{n}\sum_{d \mid (\frac{n}{f}, k_1, \dots, k_t)} \frac{n^{t}}{(n+k_{1}) \cdots (n+k_{t})} \binom{\frac{n+k_{1}}{d}}{\frac{k_{1}}{d}} \cdots \binom{\frac{n+k_{t}}{d}}{\frac{k_{t}}{d}} C_{d}(b).
\end{align*}
This completes the proof of Theorem \ref{perm thm}.
\vskip 5mm

\noindent As an illustration of the first statement, consider the
congruence $2(x_1+x_2)+3(x_3+x_4) \equiv 5$ mod $6$. Theorem \ref{perm thm} gives
the number of solutions to be $63$ which can be seen by listing all
the solutions with $x_1 \geq x_2$ and $x_3 \geq x_4$.

\vskip 2mm
\noindent
 As an illustration of the second statement, consider the
congruence $x_1+x_2+3(x_3+x_4) \equiv 1$ mod $4$. Theorem \ref{perm thm} gives
the number of solutions to be $24$ which can be seen by listing all
the solutions with $x_1 \geq x_2$ and $x_3 \geq x_4$.

\section{Strict order restricted congruence - Proof of Theorem 2}\label{proof of thm2}

\noindent \textbf{Proof of Theorem \ref{strict ord thm}. } By taking
$A=\{a, 2a, \dots, na\}$ and $q=e^{2\pi im/n}$, where $m$ is a
non-negative integer in Lemma \ref{gen fn lem}, we see that
$$\sum_{b=1}^{n}(-1)^{k}N_n(a, b) e\bigg(\frac{bm}{n}\bigg)$$
equals to the coefficient of $z^{k}$ in
$$\prod_{j=1}^{n} \bigg(1-ze\bigg(\frac{jm}{n}\bigg)\bigg).$$
By using Lemma \ref{exp prod lem}, we obtain
$$\sum_{b=1}^{n}(-1)^{k}N_n(a, b) e\bigg(\frac{bm}{n}\bigg)=(-1)^{\frac{kd}{n}}\binom{d}{\frac{kd}{n}},$$
where $d=(am,n)$. Now, by taking the inverse Fourier transform, we write
\begin{align*}
N(a, b)=&\frac{(-1)^{k}}{n}\sum_{m=1}^{n} (-1)^{\frac{kd}{n}}\binom{d}{\frac{kd}{n}} e\bigg(\frac{-bm}{n}\bigg)\\
=&\frac{(-1)^{k}}{n}\sum_{d \mid n}\sum_{\substack{m=1 \\ (am,n)=d}}^{n}(-1)^{\frac{kd}{n}}\binom{d}{\frac{kd}{n}} e\bigg(\frac{-bm}{n}\bigg)\\
=&\frac{(-1)^{k}f}{n}\sum_{d \mid \frac{n}{f}} (-1)^{\frac{kdf}{n}}\binom{df}{\frac{kdf}{n}} \sum_{\substack{m=1 \\ (m,\frac{n}{f})=d}}^{\frac{n}{f}} e\bigg(\frac{-bm}{n}\bigg)\\
=&\frac{(-1)^{k}f}{n}\sum_{d \mid \frac{n}{f}} (-1)^{\frac{kdf}{n}}\binom{df}{\frac{kdf}{n}} C_{\frac{n}{df}}(-b)\\
=& \frac{(-1)^{k}f}{n}\sum_{d \mid (\frac{n}{f}, k)} (-1)^{\frac{k}{d}}\binom{\frac{n}{d}}{\frac{k}{d}}C_{d}(b).
\end{align*}
This completes the proof of Theorem \ref{strict ord thm}.

\vskip 5mm

\subsection*{Acknowledgment}\

This work was done in October  2023. The first and second authors would like to thank the Indian Statistical Institute, Bangalore centre for providing an ideal environment to carry out this work. The second author expresses his gratitude to NBHM for financial support during the period of this work.

%\bibliographystyle{plain}    %% ???
%\bibliography{ref_crt}

\end{document}